\documentclass[11pt]{amsart}

\usepackage{geometry}
\usepackage{amssymb}
\usepackage{mathtools}
\usepackage{xcolor}
\usepackage{tikz-cd}
\usepackage[colorlinks=true,citecolor=red,linkcolor=blue,urlcolor=red]{hyperref}

\DeclareMathOperator{\Gr}{Gr}
\DeclareMathOperator{\Sing}{Sing}
\DeclareMathOperator{\length}{length}
\DeclareMathOperator{\codim}{codim}
\DeclareMathOperator{\Hess}{Hess}
\newcommand{\PP}{\mathbb{P}}
\newcommand{\CC}{\mathbb{C}}
\newcommand{\cO}{\mathcal O}
\newcommand{\cP}{\mathcal P}
\newcommand{\cH}{\mathcal H}
\newcommand{\dual}[1]{(#1)^{\vee}}

\newtheorem{thm}{Theorem}[section]
\newtheorem{lem}[thm]{Lemma}

\theoremstyle{definition}

\newtheorem{rem}[thm]{Remark}

\numberwithin{equation}{section}

\begin{document}

\title[Twelve common flex lines in a general pencil of cubics]{Twelve common flex lines in a general pencil of cubics}

\author{Jihao Liu}

\address{Department of Mathematics, Peking University, No. 5 Yiheyuan Road, Haidian District, Beijing 100871, China}
\address{Beijing International Center for Mathematical Research, Peking University, No. 5 Yiheyuan Road, Haidian District, Beijing 100871, China}
\email{liujihao@math.pku.edu.cn}

\author{Yanze Wang}
\address{Academy of Mathematics and Systems Science, Chinese Academy of Sciences, No. 55 Zhonguancun East Road, Haidian District, Beijing, 100190, China}
\email{wangyanze@amss.ac.cn}

\date{\today}

\begin{abstract}
  We prove that a general pencil of plane cubics over $\mathbb C$ has exactly $12$ common flex lines. This answers a question of Ciliberto, Miranda, and Ro\'e. The main result of this paper was obtained using generative AI, particularly ChatGPT 5.5 Pro and the Danus system.
\end{abstract}

\maketitle

\section{Introduction}\label{sec:introduction}

We work over the field of complex numbers $\CC$.

Flexes of plane curves are among the oldest objects of projective algebraic geometry. For a single smooth plane cubic, the nine flexes are governed by the group law; for a pencil, the flex tangent lines vary with the member and sweep out a curve in the dual projective plane. The geometry of this dual curve records how inflectional data degenerates and how distinct members of the pencil come to share a tangent line.

Let $\cP$ be a general pencil of plane curves of degree $d$, and let
\[
  F_{\cP,L}\subset\dual{\PP^2}
\]
denote its \emph{flex-line curve}: the closure, in the dual projective plane, of the set of lines occurring as flex tangent lines at smooth points of members of $\cP$. Here a \emph{flex tangent line} of a member $C$ at a smooth point $p$ is a line $L$ with local intersection multiplicity $I_p(L,C)\ge 3$. A \emph{common flex line} of $\cP$ is a line $L\in\dual{\PP^2}$ which is a flex tangent line at a smooth point to two distinct members of $\cP$.

Ciliberto, Miranda, and Ro\'e \cite[Propositions~1.1 and~1.2]{CMR26} computed the principal numerical invariants of $F_{\cP,L}$. For a general degree-$d$ pencil,
\begin{equation}\label{eq:cmr-degree-genus}
  \deg(F_{\cP,L})=3d(d-2),\qquad p_g(F_{\cP,L})=12d^2-39d+25,
\end{equation}
where $p_g$ denotes the geometric genus, and the hyperflex invariant is
\begin{equation}\label{eq:cmr-hyperflex}
  \deg(\cH_{(4),d})=6(d-3)(3d-2),
\end{equation}
where $\cH_{(4),d}$ is the locus of degree-$d$ plane curves admitting a line of contact order at least $4$. For $d=3$, the flex-line curve of a general cubic pencil thus has degree $9$ and geometric genus $16$, while the hyperflex invariant vanishes. Ciliberto--Miranda--Ro\'e then asked \cite[Remark~5.3]{CMR26} whether a general cubic pencil has $12$ common flex lines. The purpose of this paper is to answer this question affirmatively.

\begin{thm}\label{thm:main}
  A general pencil $\cP$ of plane cubics over $\CC$ has exactly $12$ common flex lines.
\end{thm}

The number $12$ is the singularity defect of the degree-$9$ flex-line curve: a plane curve of degree $9$ has arithmetic genus $\binom{8}{2}=28$, so \eqref{eq:cmr-degree-genus} forces a total delta-invariant equal to $28-16=12$. The content of Theorem~\ref{thm:main} is the geometric identification of that defect: for a general pencil, every singularity of the flex-line curve is an ordinary node, and each ordinary node is a common flex line.

A priori the defect could be carried by cusps, by ramification of the tangent-line map, by hyperflexes, by trisecant coincidences, or by tangent data coming from singular members of the pencil. Ruling out each of these possibilities, and matching the remaining ordinary nodes with common flex lines, is the heart of the proof.

The mechanism that makes this matching possible is the restriction of the pencil to lines.
Restricting $\cP$ to a line $L$ yields a pencil $\PP(W_L)$ of binary cubics in $\PP(R_L)\cong\PP^3$. The triple-root binary cubics form a twisted cubic $T_L$, and $L$ is a flex tangent line of a member $C$ at a smooth point $p$ precisely when $\PP(W_L)$ meets $T_L$ at $[\ell_p^3]$. Since $T_L$ has no trisecant lines and its secant variety admits a transparent stratification, the singularities of $F_{\cP,L}$ are governed by the elementary geometry of lines meeting $T_L$. The vanishing of the hyperflex invariant \eqref{eq:cmr-hyperflex} for $d=3$ removes the only stratum that could otherwise produce nonnodal singularities.

\subsection*{Outline of the proof}
The argument has three parts. In Section~\ref{sec:preliminaries} we record the geometry of lines meeting a twisted cubic in the projective space of binary cubics, and we set up the restriction of a pencil of cubics to a varying line. In Section~\ref{sec:flex-line-curve} we prove the genericity statements: for a general pencil the flex-line curve is irreducible (Lemma~\ref{lem:irreducible}), and every one of its singularities is an ordinary node whose two branches are flex data at distinct smooth points of distinct members sharing one tangent line (Lemma~\ref{lem:local-structure}). In Section~\ref{sec:proof} we combine the numerical input from \eqref{eq:cmr-degree-genus} and \eqref{eq:cmr-hyperflex} (restated in Lemmas~\ref{lem:genus} and~\ref{lem:numerical}) with the local structure to compute the genus defect and conclude that there are exactly $12$ common flex lines.

\begin{rem}
The main result of this paper was obtained using generative AI, particularly ChatGPT 5.5 Pro and the Danus system, a specialized agent built on the Rethlas system and substantially more capable of conducting fundamental mathematical research. Human verification and polishing were done afterwards. See \cite{Liu+26} and \cite{Ju+26} for detailed introductions to the Danus system and the Rethlas system, respectively. Due to the limitation of generative AI, it is possible that we have missed some related references in the literature, and we welcome any comments from experts.
\end{rem}

\subsection*{Acknowledgements}
The first author was partially supported by the National Key R\&D Program of China \#\allowbreak 2024YFA1014400. The first author would like to thank the Rethlas team, namely Haocheng Ju, Jiedong Jiang, Shurui Liu, Guoxiong Gao, Yuefeng Wang, Zeming Sun, Bin Wu, Liang Xiao, and Bin Dong, for their contributions to the development of Rethlas and its customized version used for the problem studied in this paper. The first author would like to thank Ruochuan Liu and Gang Tian for constant support and encouragement. The second author would like to thank Yifei Chen for constant support.

\section{Lines, binary cubics, and the twisted cubic}\label{sec:preliminaries}

In this section, We first recall the basic notation concerning the flex incidence curve and the flex-line curve.

\subsection*{Flex-line curve}

Write $B=\dual{\PP^2}$ for the dual plane, whose points are the lines $L\subset\PP^2$, and set $V=H^0(\PP^2,\cO_{\PP^2}(3)).$ Thus $\PP(V)$ parametrizes plane cubic curves.
Let $\Psi=\{(L,p)\in B\times\PP^2:p\in L\}$ be the point-line incidence variety.

For a two-dimensional subspace $W\subset V$ and the corresponding pencil $\cP=\PP(W)=\{C_t:t\in\PP(W)\}$, define its \emph{flex incidence curve}
\[
  \Gamma_W=\{(t,L,p)\in\PP(W)\times\Psi:I_p(L,C_t)\ge 3\}.
\]
The tangent-line projection
\[
  \lambda_W\colon\Gamma_W\longrightarrow B,\qquad (t,L,p)\longmapsto L,
\]
has image whose reduced closure is the flex-line curve $F_{\cP,L}$. Thus $\Gamma_W$ remembers the member and contact point that are forgotten by $F_{\cP,L}$. We write $\Gamma_W^\circ$ for the open locus on which $p$ is a smooth point of $C_t$.

\[
  \begin{tikzcd}[column sep=large, row sep=large]
    \Gamma_W
    \arrow[r, hook]
    \arrow[d, "\lambda_W"']
    & \cP\times\Psi
    \arrow[r, hook]
    & \PP(V)\times\Psi
    \arrow[d, "\operatorname{pr}_B\circ\operatorname{pr}_{\Psi}"] \\
    F_{\cP,L}
    \arrow[rr, hook]
    &
    & B .
  \end{tikzcd}
\]

\subsection*{Twisted cubic}
Then we record the linear-algebraic model that underlies the proof of Theorem~\ref{thm:main}.

Let $G=\Gr(2,V)$ be the Grassmannian of two-dimensional subspaces $W\subset V$. Since $\dim_\CC V=10$, we have $\dim G=16$. A point $W\in G$ determines a pencil $\cP=\PP(W)$ of plane cubics. We say a property holds for a \emph{general pencil} if it holds for all $W$ in a nonempty Zariski open subset of $G$.

For a line $L \in B$, let $R_L=H^0(L,\cO_L(3))$ be the four-dimensional space of binary cubics on $L$, and let $\rho_L\colon V\to R_L$ be the restriction map of cubic forms to $L$. The map $\rho_L$ is surjective for every $L$, with kernel
\[
  B_L=\{s\in V:s\text{ is divisible by a linear equation of }L\},
\]
of dimension $6$, so that $V/B_L\cong R_L$. As $L$ varies, the spaces $R_L$ form a rank-four vector bundle $R$ on $B$.

A nonzero binary cubic on $L$ has a triple root exactly when it is the cube of a nonzero linear form. The classes of such cubics form the twisted cubic
\[
  T_L=\{[\ell^3]:0\neq\ell\in H^0(L,\cO_L(1))\}\subset\PP(R_L)\cong\PP^3.
\]
Concretely, after a choice of coordinates $L\cong\PP^1$ and $R_L\cong H^0(\PP^1,\cO(3))$, this is the image of the map $[\alpha:\beta]\mapsto[(\alpha X+\beta Y)^3]$.

Moreover, there is a relative Grassmann bundle associated to the rank-four bundle $R$
\[
  \pi_{\Gr} : \Gr_B(2,R)\longrightarrow B.
\]
Its fiber over $L\in B$ is the Grassmannian $\Gr(2,R_L)$ of two-dimensional vector subspaces $M\subset R_L$. Thus a point of $\Gr_B(2,R)$ is a pair $(L,M)$, where
\[
  M\subset H^0(L,\cO_L(3)),\qquad \dim M=2,
\]
and its projectivization $\PP(M)$ is a projective line in $\PP(R_L)\cong\PP^3$.

Since $\rho_L\colon V\to R_L$ is surjective, a pencil $W\in G$ whose members contain no line component satisfies $W\cap B_L=0$ for every $L$, and then $\rho_L$ maps $W$ isomorphically onto a two-dimensional subspace
\[
  W_L:=\rho_L(W)\subset R_L .
\]
Let $G^\circ\subset G$ be the open subset of pencils with no member containing a line component. This is nonempty and open: the cubics with a line component form a subvariety of dimension $2+5=7$ in $\PP(V)\cong\PP^9$ (choose the line, then the residual conic), hence of codimension $2$, so a general projective line $\PP(W)$ avoids it. For $W\in G^\circ$, the assignment $L\mapsto W_L$ defines a section
\[
  \sigma_W\colon B\longrightarrow\Gr_B(2,R)
\]
of the relative Grassmann bundle of two-dimensional subspaces of $R$, and the membership $L\in F_{\cP,L}$ is governed by whether the line $\PP(W_L)\subset\PP(R_L)$ meets $T_L$.

More precisely, we have following Lemmas.

\begin{lem}\label{lem:flex-meets-twisted}
  Let $W\in G^\circ$ and $L\in B$. A member $C\in\PP(W)$ has $L$ as a flex tangent line at a smooth point $p\in L$ if and only if the binary cubic $\rho_L(C)$ has a triple root at $p$, that is, $[\rho_L(C)]\in T_L$ and $p$ is a smooth point of $C$. In particular, $L\in F_{\cP,L}$ if and only if the line $\PP(W_L)\subset\PP(R_L)$ meets $T_L$ at a class arising from a smooth point of the corresponding member.
\end{lem}

\begin{proof}
  The restriction $\rho_L(C)=C|_L$ is, as a binary cubic, the divisor $C\cap L$ counted with multiplicity. The local intersection multiplicity $I_p(L,C)$ equals the order of vanishing of $C|_L$ at $p$. Thus $I_p(L,C)\ge 3$ if and only if $C|_L$ vanishes to order $3$ at $p$, that is, has a triple root at $p$, which is exactly the condition $[C|_L]\in T_L$ with triple root $p$. Since $\rho_L|_W$ is injective for $W\in G^\circ$, the line $\PP(W_L)$ is the image of $\PP(W)$, and $\PP(W_L)$ meets $T_L$ precisely when some member of $\PP(W)$ restricts to a triple-root binary cubic on $L$. The smoothness of $C$ at $p$ is the requirement that $p$ be a smooth flex point rather than a singular point of $C$.
\end{proof}

\begin{lem}\label{lem:twisted-cubic}
  The twisted cubic $T_L\subset\PP(R_L)$ has the following properties.
  \begin{enumerate}
    \item No projective line in $\PP(R_L)$ contains three distinct points of $T_L$; that is, $T_L$ has no trisecant lines.
    \item A projective line $M\subset\PP(R_L)$ that meets $T_L$ in a single reduced point and is not tangent to $T_L$ is a smooth point of the hypersurface $\Sigma_L\subset\Gr(2,R_L)$ of lines meeting $T_L$.
    \item A projective line $M\subset\PP(R_L)$ that is a secant through two distinct points of $T_L$ is an ordinary node of $\Sigma_L$: near $M$ the locus $\Sigma_L$ is the transverse union of two smooth hypersurface branches, one for each of the two secant points.
    \item The locus of tangent lines to $T_L$ has codimension $3$ in $\Gr(2,R_L)$, and the locus of secant lines through two distinct points of $T_L$ has codimension $2$.
  \end{enumerate}
\end{lem}

\begin{proof}
  We prove (1). Suppose three distinct points $[\ell_1^3],[\ell_2^3],[\ell_3^3]$ of $T_L$ were collinear. After a change of coordinate on $L\cong\PP^1$ we may take $\ell_1=X$, $\ell_2=Y$, $\ell_3=X+Y$, so the three binary cubics are $X^3$, $Y^3$, and $(X+Y)^3$. A linear relation $aX^3+bY^3+c(X+Y)^3=0$ forces, by inspecting the coefficients of $X^2Y$ and $XY^2$ in $(X+Y)^3=X^3+3X^2Y+3XY^2+Y^3$, that $c=0$, and then $a=b=0$. Thus $X^3,Y^3,(X+Y)^3$ are linearly independent in $R_L$, so no projective line contains all three points. This proves (1).

  The local statements (2) and (3) are the standard analytic description of the secant variety of a twisted cubic; we recall the argument. The hypersurface $\Sigma_L\subset\Gr(2,R_L)$ of lines meeting $T_L$ is the image of the projection from the universal secant incidence
  \[
    \{(M,x):x\in M\cap T_L\}\subset\Gr(2,R_L)\times T_L.
  \]
  At a line $M$ meeting $T_L$ in a single reduced point $x$ at which $M$ is not tangent, this incidence is smooth and maps isomorphically to its image near $M$; the condition of meeting $T_L$ near $x$ is a single smooth equation, and $\Sigma_L$ is smooth at $M$, proving (2). At a secant $M$ through two distinct points $x_1\neq x_2$, the incidence has two distinct smooth sheets near $M$, one over each $x_i$; each sheet is locally the smooth hypersurface ``$M$ continues to meet $T_L$ near $x_i$''. The two sheets have distinct tangent hyperplanes in $\Gr(2,R_L)$ because the conditions of passing near two distinct points of the twisted cubic are independent. Hence near $M$ the locus $\Sigma_L$ is the transverse union of two smooth branches, an ordinary node, proving (3).

  For (4), the tangent lines to $T_L$ form a one-parameter family, parametrized by the point of tangency, so the tangent locus has dimension $1$ in $\Gr(2,R_L)$, which has dimension $4$; hence its codimension is $3$. The secant lines through two distinct points of $T_L$ are parametrized by the unordered pairs of distinct points of $T_L$, a two-dimensional family, so the secant locus has dimension $2$, hence codimension $2$. This proves (4).
\end{proof}

\section{The flex-line curve of a general pencil}\label{sec:flex-line-curve}

In this section, we prove that for a general pencil of plane cubics the flex-line curve is irreducible and has only ordinary nodes, each corresponding to a common flex line. Throughout, $V$, $G$, $G^\circ$, $B$, $R$, $T_L$, $W_L$, and $\sigma_W$ are as in Section~\ref{sec:preliminaries}.

\subsection*{Irreducibility}

\begin{lem}\label{lem:irreducible}
  For a general pencil $\cP=\PP(W)$ of plane cubics, the flex-line curve $F_{\cP,L}$ is irreducible.
\end{lem}

\begin{proof}
  Consider the smooth-flex incidence variety
  \[
    X^\circ=\{(L,p,[s]):L\in B,\ p\in L,\ [s]\in\PP(V),\ I_p(L,\{s=0\})=3,\ p\text{ smooth on }\{s=0\}\}.
  \]
  The pairs $(L,p)$ with $p\in L$ form a smooth threefold; over each such pair, the condition $I_p(L,\{s=0\})\ge 3$ cuts out a linear subspace of codimension $3$ in $V$ (the vanishing of the coefficients of $1$, $x$, $x^2$ in $s|_L$, for a coordinate $x$ on $L$ centered at $p$). Hence $X^\circ$ is an open subset of a $\PP^6$-bundle over a smooth threefold, so it is smooth and irreducible of dimension $9$. The universal incidence over the Grassmannian,
  \[
    \mathcal X^\circ=\{(W,L,p,[s])\in G\times X^\circ:[s]\in\PP(W)\},
  \]
  fibers over $X^\circ$ with fibers $\PP(V/\CC s)$ of dimension $8$, so $\mathcal X^\circ$ is irreducible.

  Let $\pi\colon\mathcal X^\circ\to G$ and $\tau\colon\mathcal X^\circ\to B$ be the two projections. The fiber of $\pi$ over $W\in G^\circ$ is the flex incidence curve curve $\Gamma_W^\circ$, and $\tau$ restricted to that fiber is the tangent-line map onto a dense subset of $F_{\cP,L}$. To see that the generic fiber of $\pi$ is irreducible, observe that $\mathcal X^\circ$ is irreducible and dominates $G$: a general pencil contains a smooth cubic, which has nine flexes, so the flex incidence curve curve is nonempty.
  The fibre of $ \mathcal{X}^\circ \to G $ at $ W $ is the inverse image of corrosponding projective line  $ \mathbb{P}(W) $ under $ \mathcal{X}^\circ \to \mathbb{P}(V) $. Therefore by \cite[Theorem~3.3.1]{Laz04}, the fiber of $\pi$ over a general point of $G$ is irreducible.
  The image of the fibre in $ B $ is dense in $F_{\cP,L}$, so $F_{\cP,L}$ is irreducible as well.
\end{proof}

\subsection*{Local structure}

We now establish the precise local form of the singularities. The next lemma combines a dimension count over the parameter space of pencils with the explicit local computations of the tangent-line map.

\begin{lem}\label{lem:transversality}
  For a general two-dimensional subspace $W\subset V$, the following hold.
  \begin{enumerate}
    \item $W$ contains no nonzero cubic divisible by a line, so $W\in G^\circ$.
    \item The tangent-line map from the flex incidence curve curve to $B$ is unramified at every smooth flex datum.
    \item If two smooth flex data of $\PP(W)$ share a tangent line $L$, then their flex points on $L$ are distinct, no third smooth flex datum has tangent line $L$, and the two local image branches in $B$ meet transversely at $L$.
  \end{enumerate}
\end{lem}

\begin{proof}
  Statement (1) was established in Section~\ref{sec:preliminaries}: the line-component locus has codimension $2$ in $\PP(V)$, so a general $W$ lies in $G^\circ$. Fix $W\in G^\circ$, so that $\rho_L$ maps $W$ isomorphically onto $W_L\subset R_L$ for every $L$.

  For (2), fix a smooth flex datum $(L,p,C)$. Choose affine coordinates near $p$ with $p=(0,0)$ and $L=\{y=0\}$, in which $C$ has local equation $y=x^3+(\text{terms of order}\ge 4)$ after multiplying by a unit. Let $G_0$ be a cubic form spanning $W$ together with $C$, representing a first-order direction of the pencil at $C$, and put $g(x)=G_0(x,0)$. To first order in a parameter $t$, the nearby member of the pencil has local graph
  \[
    y=x^3-t\,g(x)+(\text{terms divisible by }t^2).
  \]
  If $s(t)$ denotes the $x$-coordinate of the nearby flex point, the inflection condition gives $6s(t)-t\,g''(s(t))+O(t^2)=0$, so $s(0)=0$. The nearby flex tangent line has slope-derivative $-g'(0)$ and intercept-derivative $-g(0)$ in $t$. Hence the differential of the tangent-line map at $(L,p,C)$ vanishes only if $g(0)=g'(0)=0$, that is, only if the first-order direction $G_0$ restricts to $L$ with vanishing order at least $2$ at $p$. Since $C$ itself restricts to $L$ with vanishing order $3$ at $p$, this condition forces $W$ into the subspace $A_2(L,p)=\rho_L^{-1}(B_2(p))$, where $B_2(p)\subset R_L$ is the two-dimensional space of binary cubics vanishing to order at least $2$ at $p$. For a fixed pair $(L,p)$, the set of such $W$ is $\Gr(2,A_2(L,p))$, of dimension $2(8-2)=12$. Letting $(L,p)$ vary over its three-dimensional parameter space, the union of these loci has dimension at most $15<16=\dim G$, so a general $W$ avoids it. This proves (2).

  For (3), fix $L$. By Lemma~\ref{lem:flex-meets-twisted}, smooth flex data of $\PP(W)$ over $L$ correspond to points $p\in L$ for which the projective line $\PP(W_L)$ contains the point $[\ell_p^3]\in T_L$, where $\ell_p$ is a linear form vanishing at $p$. By Lemma~\ref{lem:twisted-cubic}(1), $\PP(W_L)$ contains at most two distinct points of $T_L$, so no line $L$ carries three smooth flex data. If two smooth flex data over $L$ had the same point $p$ but came from two distinct cubics in $W$, then $\rho_L(W)$ would lie in the one-dimensional space $B_3(p)$ of binary cubics with a triple root at $p$, contradicting $\dim W_L=2$. Hence two distinct flex data over $L$ have distinct flex points $p_1\neq p_2$.

  It remains to prove transversality of the two image branches. Choose a coordinate $x$ on $L$ with $p_1$ at $x=0$ and $p_2$ at $x=a$, $a\neq 0$. The spaces $B_3(p_1)$ and $B_3(p_2)$ are spanned by $x^3$ and $(x-a)^3$ respectively. By the first-order computation in part (2), the tangent vector in $B$ of the branch through $p_1$ is proportional to
  \[
    \bigl(-\bigl((x-a)^3\bigr)'\big|_{x=0},\ -\bigl((x-a)^3\bigr)\big|_{x=0}\bigr)=(-3a^2,\,a^3),
  \]
  and the tangent vector of the branch through $p_2$ is proportional to
  \[
    \bigl(-(x^3)'\big|_{x=a},\ -(x^3)\big|_{x=a}\bigr)=(-3a^2,\,-a^3).
  \]
  These two vectors are not proportional, since $a\neq 0$. Hence the two image branches meet transversely. This proves (3).
\end{proof}

\begin{lem}[Local structure of the flex-line curve]\label{lem:local-structure}
  For a general pencil $\cP=\PP(W)$ of plane cubics, every singular point of $F_{\cP,L}$ is an ordinary node. The two branches over each node are smooth flex data $(L,p_1,C_1)$ and $(L,p_2,C_2)$ with the same line $L$, distinct flex points $p_1\neq p_2$, and distinct members $C_1\neq C_2$. Branches arising as limits at the nodes of singular cubic members of $\cP$ are smooth immersed branches of $F_{\cP,L}$ and do not contribute to $\Sing(F_{\cP,L})$. Consequently the singular points of $F_{\cP,L}$ are in bijection with the common flex lines of $\cP$.
\end{lem}

\begin{proof}
  Fix a general $W$, lying in $G^\circ$ and satisfying the conclusions of Lemma~\ref{lem:transversality}. Consider the section $\sigma_W\colon B\to\Gr_B(2,R)$, $\sigma_W(L)=W_L$, and the relative locus $\Sigma\subset\Gr_B(2,R)$ whose fiber over $L$ is the hypersurface $\Sigma_L$ of lines in $\PP(R_L)$ meeting $T_L$. By Lemma~\ref{lem:flex-meets-twisted}, away from the loci excluded below, $F_{\cP,L}=\sigma_W^{-1}(\Sigma)$.

  The fiberwise strata of Lemma~\ref{lem:twisted-cubic} assemble into relative locally closed strata $\Sigma^{\mathrm{sm}}$, $\Sigma^{\mathrm{sec}}$, and $\Sigma^{\mathrm{tan}}$ of $\Gr_B(2,R)$, parametrizing lines that meet $T_L$ at one reduced non-tangent point, at two distinct points, and tangentially, of relative codimensions $1$, $2$, and $3$ respectively. Consider the evaluation morphism
  \[
    E\colon G^\circ\times B\longrightarrow\Gr_B(2,R),\qquad (W,L)\longmapsto W_L.
  \]
  For fixed $L$, the restriction map $\rho_L\colon V\to R_L$ is surjective, so the infinitesimal variations of $W\subset V$ realize all infinitesimal variations of the image plane $W_L\subset R_L$; thus $E$ is submersive onto the relative Grassmann directions. By generic smoothness applied to $E^{-1}(Z)\to G^\circ$ for each smooth stratum $Z$, a general $W$ has $\sigma_W$ transverse to $\Sigma^{\mathrm{sm}}$ and $\Sigma^{\mathrm{sec}}$, and avoids $\Sigma^{\mathrm{tan}}$ altogether, because $\dim B=2<3=\codim\Sigma^{\mathrm{tan}}$.

  Consequently, for a general pencil, $F_{\cP,L}=\sigma_W^{-1}(\Sigma)$ is smooth away from the finite transverse inverse image $\sigma_W^{-1}(\Sigma^{\mathrm{sec}})$, and at each point of $\sigma_W^{-1}(\Sigma^{\mathrm{sec}})$ it has an ordinary node by Lemma~\ref{lem:twisted-cubic}(3). The tangent stratum is avoided, so there are no cusps or stationary-flex singularities; this is consistent with the vanishing of the hyperflex invariant for $d=3$ in \eqref{eq:cmr-hyperflex}. There are no triple coincidences by Lemma~\ref{lem:twisted-cubic}(1). The two image branches at a secant meet transversely by transversality to $\Sigma^{\mathrm{sec}}$, in accordance with Lemma~\ref{lem:transversality}(3).

  We next show that, for a general pencil, the two branches over each node are flexes at smooth points of two distinct members. Consider the bad locus where one of the two secant points of $\PP(W_L)\cap T_L$ arises from a cubic singular at its triple-root point. Such a configuration is given by a triple $(L,p,q)$ with $p,q\in L$, $p\neq q$, a cubic $C_1$ whose restriction to $L$ has a triple root at $p$ and which is singular at $p$, and a cubic $C_2$ whose restriction to $L$ has a triple root at $q$. The triple $(L,p,q)$ has dimension $4$; the class $[C_1]$ has dimension $5$, because being singular at $p$ together with the triple-root condition at $p$ imposes four independent linear conditions on $\PP(V)\cong\PP^9$; and $[C_2]$ has dimension $6$, the ordinary triple-root condition imposing three independent linear conditions. These data determine at most one pencil $\PP\langle C_1,C_2\rangle$, so their image in $G$ has dimension at most $4+5+6=15<16$. Hence a general pencil has no node of $F_{\cP,L}$ at which one branch comes from a singular point of its cubic; the case where both branches are singular is contained in the same proper closed locus.

  Next we account for the singular members of the pencil. The discriminant hypersurface in $\PP(V)$ has singular locus of codimension at least $2$ (cubics with worse than a single ordinary node), so a general pencil meets the discriminant transversely and only in nodal cubics. We must show that the flex limits at such nodes are smooth immersed branches that do not create singularities of $F_{\cP,L}$, and that they do not coincide with other branches. A transverse one-parameter smoothing of a node of a plane cubic can be written analytically as
  \[
    f_t(x,y)=xy+a x^3+b y^3+t=0,\qquad ab\neq 0,
  \]
  with homogenization $F_t(X,Y,Z)=XYZ+aX^3+bY^3+tZ^3$. Its Hessian determinant is
  \[
    \Hess(F_t)=216ab\,t\,XYZ-6aX^3-6bY^3+2XYZ-6tZ^3 .
  \]
  On the surface $f_t=0$ one has $t=-xy-ax^3-by^3$; substituting and discarding terms divisible by $x^2y^2$ gives
  \[
    \Hess(F_t)\big|_{f_t=0}=8xy+(\text{terms of order}\ge 4).
  \]
  In the blow-up chart $y=ux$, the pullback is $x^2u\bigl(8+(\text{terms vanishing at }x=0)\bigr)$. After removing the multiplicity-two exceptional contribution, the proper transform of the flex divisor is $\{u=0\}$, meeting the exceptional divisor $\{x=0\}$ transversely in the tangent direction $y=0$; the other chart gives the second transverse intersection at $x=0$. Along the first branch, a point $(x,0)$ maps under the tangent-line map to the line of affine direction $[f_x:f_y]=[3ax^2:x]$, whose first-order variation in $B$ is nonzero, and similarly for the second branch. Hence each nodal-boundary branch is a smooth immersed branch of $F_{\cP,L}$.

  It remains to exclude coincidences involving nodal-boundary branches. A coincidence between a nodal tangent line and a smooth flex datum of another member is parametrized by the node $p$, the tangent line $L\ni p$, the nodal cubic $C_1$, a smooth flex point $q\in L$, and a smooth-flex cubic $C_2$; here $L$ singular at $p$ with a prescribed nodal tangent direction along $L$ imposes, on cubics, four independent linear conditions, so $[C_1]$ moves in dimension $5$, while $[C_2]$ moves in dimension $6$. Counting $(p,L,q)$ as dimension $2+1+1=4$, the image in $G$ has dimension at most $4+5+6=15$. A coincidence between two nodal tangent data sharing a line is bounded similarly by $4 + 5 + 5 = 14$. Both images are proper closed subsets of $G$, so a general pencil avoids them.

  Removing all of the proper bad loci above, we obtain a nonempty open set of pencils for which: the normalization $\widetilde F\to F_{\cP,L}$ is generically one-to-one; every singular point of $F_{\cP,L}$ is an ordinary node coming from exactly two smooth flex data with a common tangent line at distinct points; and nodal-boundary branches are smooth, immersed, and isolated from all other branches. Finally, if the two smooth flex data over a node came from the same member $C$, then the line $L$ would meet $C$ with multiplicity at least $3$ at each of two distinct smooth points, hence with total intersection multiplicity at least $6$; this is impossible for a cubic $C$ not containing $L$, and members of $\cP$ contain no line component. Therefore the two branches over each node come from two distinct members $C_1\neq C_2$, so each node is a common flex line. Conversely, by Lemma~\ref{lem:flex-meets-twisted} a common flex line $L$ makes $\PP(W_L)$ a secant of $T_L$ through two distinct points; the generality conditions force it to lie in $\sigma_W^{-1}(\Sigma^{\mathrm{sec}})$, hence to be an ordinary node. Thus the singular points of $F_{\cP,L}$ are in bijection with the common flex lines of $\cP$.
\end{proof}

\section{Proof of the main theorem}\label{sec:proof}

We first record the numerical input from \cite[Propositions~1.1 and~1.2]{CMR26}, specialized to cubics.

For a general pencile, the flex incidence curve $ \Gamma_W $ is the normalization of the flex-line curve.

\begin{lem}\label{lem:normalization}
  For a general pencil $\cP=\PP(W)$, the curve $\Gamma_W$ is irreducible and
  normal, and
  \[
    \lambda_W\colon\Gamma_W\longrightarrow F_{\cP,L}
  \]
  is the normalization morphism.
\end{lem}

\begin{proof}
  The open subcurve $\Gamma_W^\circ$ is irreducible by
  Lemma~\ref{lem:irreducible}, hence so is its closure $\Gamma_W$. By the
  universal-flex construction of \cite[Section~11.3.2]{EH16}, $\Gamma_W$ is
  the zero locus of a section of a rank-three vector bundle on
  $\PP(W)\times\Psi$. Moreover,
  \cite[Section~11.3.3 and Exercises~11.33--11.36]{EH16} shows that, for a
  general pencil, this complete incidence curve is smooth. Thus $\Gamma_W$
  is normal and projective.

  Lemma~\ref{lem:local-structure} shows that $\lambda_W$ is generically
  one-to-one. It is finite: projectivity makes it proper, and its fiber over
  $L\in B$ is identified with the finite scheme
  $\PP(W_L)\cap T_L$, since the twisted cubic $T_L$ contains no line.
  Therefore $\lambda_W$ is finite and birational. Since its source is
  normal, it is the normalization morphism of $F_{\cP,L}$.
\end{proof}

Then we can compute the genera.

\begin{lem}\label{lem:genus}
  For a general pencil $\cP=\PP(W)$ of plane cubics,
  \[
    p_a(F_{\cP,L})=28,
    \qquad
    p_g(F_{\cP,L})=16.
  \]
\end{lem}

\begin{proof}
  By \eqref{eq:cmr-degree-genus}, or equivalently \cite[Proposition~1.2]{CMR26}, the flex-line curve has degree $9$. Hence
  \[
    p_a(F_{\cP,L})=\binom{9-1}{2}=28.
  \]
  On the other hand, the genus computation for the flex incidence curve in \cite[Section~11.3.2]{EH16} gives
  \[
    g(\Gamma_W)=12d^2-39d+25=16
    \qquad(d=3).
  \]
  Since $\Gamma_W$ is the normalization of $F_{\cP,L}$ by
  Lemma~\ref{lem:normalization}, this yields $p_g(F_{\cP,L})=16$.
\end{proof}

We also record that the hyperflex contribution vanishes in degree $3$.

\begin{lem}\label{lem:numerical}
  No member of a general pencil of plane cubics has a flex line of contact order at least $4$.
\end{lem}

\begin{proof}
  By \eqref{eq:cmr-hyperflex}, the hyperflex invariant is $6(d-3)(3d-2)$ \cite[Proposition~1.1]{CMR26}, which vanishes when $d=3$.
\end{proof}

Finally we show the main theorem.

\begin{proof}[Proof of Theorem~\ref{thm:main}]
  Let $\cP=\PP(W)$ be a general pencil of plane cubics, general enough that the conclusions of Lemmas~\ref{lem:irreducible}, \ref{lem:local-structure}, and \ref{lem:numerical} all hold simultaneously; this is possible since each conclusion holds on a nonempty Zariski open subset of $G$, and a finite intersection of such subsets is again nonempty and open.

  By Lemma~\ref{lem:irreducible}, $F=F_{\cP,L}$ is an irreducible plane curve, and by Lemma~\ref{lem:numerical} it has degree $9$. Therefore its arithmetic genus is
  \begin{equation}\label{eq:pa}
    p_a(F)=\binom{9-1}{2}=\binom{8}{2}=28 .
  \end{equation}
  By Lemma~\ref{lem:normalization}, $\lambda_W\colon\Gamma_W\longrightarrow F.$ is the normalization. Thus, for $q\in\Sing(F)$, its local delta-invariant is
  \[
    \delta_q=
    \length_\CC\bigl((\lambda_{W*}\cO_{\Gamma_W}/\cO_F)_q\bigr).
  \]
  The short exact sequence
  \[
    0\longrightarrow\cO_F
    \longrightarrow\lambda_{W*}\cO_{\Gamma_W}
    \longrightarrow\lambda_{W*}\cO_{\Gamma_W}/\cO_F
    \longrightarrow 0
  \]
  yields, on taking Euler characteristics,
  \begin{equation}\label{eq:delta-sum}
    \sum_{q\in\Sing(F)}\delta_q=p_a(F)-p_g(F).
  \end{equation}
  By Lemma~\ref{lem:numerical}, $p_g(F)=16$, so combining \eqref{eq:pa} and \eqref{eq:delta-sum} gives
  \[
    \sum_{q\in\Sing(F)}\delta_q=28-16=12 .
  \]

  By Lemma~\ref{lem:local-structure}, every singular point of $F$ is an ordinary node, which has delta-invariant $1$, and the singular points of $F$ are in bijection with the common flex lines of $\cP$. Writing $N$ for the number of common flex lines, we therefore have
  \[
    N=\sum_{q\in\Sing(F)}\delta_q=12 .
  \]
  Hence a general pencil of plane cubics has exactly $12$ common flex lines.
\end{proof}

\end{document}